\documentclass[hidelinks,onefignum,onetabnum]{siamart251216}
\usepackage{amssymb}    
\usepackage{booktabs}

\ifpdf
\hypersetup{
  pdftitle={Uniform convergence of diffusion synthetic acceleration for heterogeneous slab transport},
  pdfauthor={M. Schlottbom}
}
\fi

\theoremstyle{plain}
\newtheorem{algorithm_self}[theorem]{Algorithm}
\newtheorem{remark}[theorem]{Remark}

\newcommand{\D}{\mathcal{D}}
\newcommand{\PP}{\mathcal{P}}
\newcommand{\QQ}{\mathcal{Q}_r}
\newcommand{\LL}{\Lambda}
\newcommand{\Wp}{\mathbb{W}^+}
\newcommand{\Wo}{\mathbb{W}^+_1}
\newcommand{\Wph}{\mathbb{W}^+_h}
\newcommand{\Woh}{\mathbb{W}^+_{1,h}}
\newcommand{\eps}{\varepsilon}
\newcommand{\dmu}{\,\mathrm{d}\mu}
\newcommand{\dz}{\,\mathrm{d}z}

\headers{Uniform convergence of diffusion acceleration}{M. Schlottbom}
\title{Uniform convergence of diffusion synthetic acceleration for heterogeneous slab transport}
\author{%
    Matthias Schlottbom%
    \thanks{Department of Applied Mathematics, University of Twente, P.O. Box 217, 7500 AE Enschede, The Netherlands.  
    {\small\textit{e-mail}:} {\small \texttt{m.schlottbom@utwente.nl}},
    {\small ORCID:} {\small \texttt{0000-0002-2527-6498}}}}

\begin{document}
\maketitle

\begin{abstract}
The diffusion synthetic accelerated (DSA) source iteration is a standard solver for the
radiative transfer equation. Using Fourier analysis, a convergence rate \(\rho_\infty(c)\le0.2247\,c\) with
\(c\) the maximum ratio of scattering to total cross section has been established for an infinite homogeneous medium.
For slab geometry with inflow boundary conditions and arbitrary bounded cross sections we prove
that the spectral radius of the DSA iteration is at most \(\rho_\infty(c)\).
We show that this convergence rate carries over to a variational discretization of the DSA iteration on
every conforming tensor-product Galerkin space whose angular factor contains the constants.
The analysis rests on an exact min--max characterization of the spectral radius. From it we derive a
checkable sufficient condition for a given rate. We verify this condition using a suitable
energy-stable projection, a weighted angular average whose weight is chosen so that the
condition holds with the rate \(\rho_\infty(c)\). Since the projection maps discrete
spaces into discrete spaces, the argument applies verbatim to the discrete iteration. As a
consequence, the condition number of the preconditioned system is uniformly bounded by
\(1+0.29\,c\), which implies rapid convergence of the preconditioned conjugate gradients method.
\end{abstract}


% REQUIRED

% \noindent\textbf{Key words.} radiative transfer, diffusion synthetic acceleration, source
% iteration, preconditioning, convergence rate, even-parity formulation

% \smallskip
% \noindent\textbf{AMS subject classifications.} 65F08, 65F10, 65N12, 65N30, 82D75

\section{Introduction}\label{sec:intro}

We consider the monoenergetic radiative transfer equation in slab geometry,
\begin{align}
\mu\partial_z\psi(z,\mu)+\sigma_t(z)\psi(z,\mu)
&=\sigma_s(z)\,\phi(z)+q(z,\mu)
&&\text{in }\D=(0,Z)\times(-1,1),\label{eq:rte}\\
\psi&=g &&\text{on }\Gamma_-,\label{eq:bc}
\end{align}
for the angular flux \(\psi\), where
\[
\phi(z)=\frac12\int_{-1}^1\psi(z,\mu')\dmu'
\]
is the scalar flux and \(\Gamma_-=\{(0,\mu):\mu>0\}\cup\{(Z,\mu):\mu<0\}\) is the inflow
boundary. Problems of this type arise in neutron transport, heat transfer and medical imaging
\cite{Arridge_2009,CaseZweifel67,Modest}, and slab geometry is the classical setting for their
analysis \cite{AdamsLarsen02}.

The basic iterative solver for \eqref{eq:rte}--\eqref{eq:bc} is the source iteration
\cite{AdamsLarsen02,MarchukLebedev86}, which lags the scattering term. Given a scalar flux
\(\phi^{(n)}\), one solves the transport problem without scattering, the sweep,
\begin{equation}\label{eq:sweep}
\mu\partial_z\psi^{(n+1/2)}+\sigma_t\psi^{(n+1/2)}=\sigma_s\phi^{(n)}+q
\ \ \text{in }\D,\qquad \psi^{(n+1/2)}=g\ \ \text{on }\Gamma_-,
\end{equation}
and sets \(\phi^{(n+1/2)}=\frac12\int_{-1}^1\psi^{(n+1/2)}\dmu\). Each sweep reduces the
error of the scalar flux, in a suitable norm, at least by the factor
$c=\sup_{z\in(0,Z)}\sigma_s(z)/\sigma_t(z)$.
In optically thick media \(c\) is close to one, and the source iteration converges slowly. 

Diffusion synthetic acceleration (DSA) is a stationary linear iteration for the scalar flux that corrects each sweep by the solution of a
diffusion problem, see \cite{AdamsLarsen02,LarsenMorel2010} for reviews.
Motivated by the diffusion limit of transport \cite{EggerSchlottbom2014,Larsen_1974}, one solves
\begin{equation}\label{eq:dsa}
-\partial_z\Big(\tfrac1{3\sigma_t}\partial_zF^{(n+1)}\Big)+\sigma_aF^{(n+1)}
=\sigma_s\big(\phi^{(n+1/2)}-\phi^{(n)}\big)\ \ \text{in }(0,Z),
\end{equation}
with \(\sigma_a=\sigma_t-\sigma_s\), subject to suitable boundary conditions, and sets
\begin{equation}\label{eq:update}
\phi^{(n+1)}=\phi^{(n+1/2)}+F^{(n+1)}.
\end{equation}

\paragraph{Convergence theory} The classical convergence theory for DSA is Fourier analysis
for an infinite homogeneous medium
\cite{AdamsLarsen02,GelbardHageman1969,Larsen1982}, see also \cite{LarsenMorel2010} for a more recent application.
There, every Fourier mode \(e^{i\kappa z}\) is an eigenfunction of the error propagation operator, and its eigenvalue,
the symbol, is \cite[(2.50)]{AdamsLarsen02}
\begin{equation}\label{eq:alsymbol}
\omega_c(\lambda)=\frac{3c}{\lambda^2+3(1-c)}
\left[\left(\frac{\lambda^2}{3}+1\right)\frac{\arctan\lambda}{\lambda}-1\right],
\qquad \lambda=\frac{\kappa}{\sigma_t}.
\end{equation}
The diffusion correction removes the slowly varying modes and the sweep damps the rapidly oscillating ones. 
The spectral radius, which we denote by \(\rho_\infty\) since it refers to the infinite medium, is
\(\rho_\infty(c)=\sup_\lambda\omega_c(\lambda)\le0.2247\,c\).
The least damped modes have wavelengths of a few mean free paths. 
This result explains the observed robustness of DSA in scattering dominated media if Fourier techniques apply.

Ashby et al.\ \cite{AshbyEtAl1995} derived DSA algebraically for
diamond-differenced discrete ordinates on finite slabs with nonconstant coefficients,
 and proved that the preconditioned matrix converges to the identity in the
thick, the thin, and the asymptotic diffusion limit. As they point out, these limits predict
no convergence rate and do not explain the convergence observed for problems of moderate
thickness \cite{AshbyEtAl1995}. 
For high-order discontinuous Galerkin discretizations, \cite{Haut2020} proves that the interior penalty DSA preconditioned operator is a perturbation
of the identity of the order of the mean free path, so that the iteration converges rapidly
for optically thick problems on a fixed mesh.
A quantitative convergence rate for upwind discontinuous Galerkin discretizations on polytopic meshes
with constant coefficients and vacuum inflow is established in \cite{CallooEvansMadiotPryer2026b}.
Their rate requires \(\max_K p^2/(\sigma_t h_K)\) to remain bounded by a fixed constant
\cite[(49), Rem.~4.7]{CallooEvansMadiotPryer2026b}, where $p$ is the polynomial degree and $h_K$ the local mesh size,
i.e., the cells have to be optically thick. It is therefore uniform in the diffusive limit on a fixed mesh. The bounds below hold on
every mesh and are in this respect complementary.

\paragraph{Applications and consistent discretization} 
Despite the limited analytical
results, DSA is used routinely in large scale transport calculations \cite{CallooEtAl2023,Haut2020,SouthworthHolecHaut2020}.
It is well-known that the discretization of \eqref{eq:dsa} must be consistent with the discretization of the sweep,
otherwise the acceleration may be lost or the iteration may even diverge in optically thick cells
\cite{AdamsLarsen02,Alcouffe1977,Larsen1982}.
For the even-parity form, a consistent diffusion discretization
is obtained directly \cite{MorelMcGhee1995}. In several dimensions and on
unstructured meshes, fully consistent diffusion discretizations are expensive \cite{Warsa2002}
and may lead to singular matrices \cite{Brown_1995}, which motivated partially consistent
schemes that retain only the scalar flux \cite{Ragusa2010,Wang2010}. These are cheaper and
unconditionally stable in the thin and in the thick limit, but their effectiveness can degrade,
and for layered media with strong material discontinuities the rate approaches the rate \(c\) of
the unaccelerated iteration \cite{Wang2010}. Loss of effectiveness at material discontinuities
is also observed for consistent schemes in several dimensions, where DSA is therefore commonly
used as a preconditioner for Krylov methods \cite{Warsa_2004}.

A different route to consistency is variational \cite{PS2020,DPS2022,BardinSchlottbom2025}.
Starting from variational formulations of the transport equation,
the diffusion correction is obtained by Galerkin projection, so that consistency holds by construction.
The rate proved in these works is \(c\),
whereas the observed rates are robust in the diffusive regime \cite{PS2020}.

\paragraph{Approach and main results} 
In this paper we analyze the DSA iteration in slab
geometry with inflow boundary conditions, for arbitrary bounded heterogeneous cross sections,
both for the continuous iteration and for its consistent variational discretization. 
We obtain (Theorem~\ref{thm:sharp}) that, if \(\sigma_t-\sigma_s\ge\gamma>0\), then 
\begin{align}\label{eq:sharp}
\rho(G)\le\rho_\infty(c)\qquad\text{and}\qquad\rho(G_h)\le\rho_\infty(c),
\end{align}
where \(\rho\) denotes the spectral radius, \(G\) is the error propagation operator of the DSA
iteration, and \(G_h\) that of the discrete iteration of
\cite{PS2020} on any conforming tensor-product space \(V_h\otimes Q_N\) whose angular factor \(Q_N\) contains the constants. 
These bounds are independent of the thickness of the slab, the mesh size and
the angular resolution. The infinite-medium value \(\rho_\infty(c)\) is therefore a uniform bound for
heterogeneous slabs with inflow boundary conditions and for their discretizations.
Moreover, the error of the scalar flux is reduced in the corresponding norm by 
at least the factor \(\rho_\infty(c)\) per DSA step.
Numerical experiments indicate that the bound is approached by optically thick heterogeneous slabs.
In particular, the condition number of the DSA preconditioned system is bounded by
\(1+0.29\,c\), for the continuous and the discrete iteration.
As a consequence, conjugate gradients with the DSA preconditioner are guaranteed to converge rapidly,
see also the numerical experiments in \cite{FaberManteuffel1989}.

The starting point for our analysis is that the DSA iteration for the scalar flux is
self-adjoint in a natural energy inner product, which gives an exact min--max characterization
of its spectral radius. From this characterization we derive a sufficient condition for a given
convergence rate. The verification of this condition requires an energy stable projection onto
the diffusion subspace. Since plain angular averages are not regular enough, we use weighted
angular averages instead. We show that the resulting projection also maps discrete energy
spaces into discrete diffusion spaces, so that the proofs carry over to the discrete iteration.
The quotient in the min--max characterization is not monotone under Galerkin projection, so
that the continuous estimate does not automatically imply the discrete one.
The proof of \eqref{eq:sharp} is technical. We provide also a much simpler construction, 
which gives the rate \(c/4\) (Remark~\ref{rem:half}).

The paper is organized as follows. Section~\ref{sec:prelim} introduces the even-parity
formulation, the variational DSA iteration and its discretization. Section~\ref{sec:it}
identifies the error propagation operator, characterizes its spectrum by a min--max principle
and derives from it a primal criterion for a given rate. Section~\ref{sec:main} constructs the
weighted angular average that verifies this criterion and proves the main result.
Section~\ref{sec:num} provides supporting numerical results, and Section~\ref{sec:concl} concludes.

\section{Variational setting}\label{sec:prelim}

We write \((v,w)=\int_{-1}^1\!\int_0^Zvw\dz\dmu\) for the inner product of \(L^2(\D)\) and
\(\|w\|^2_\sigma=(\sigma w,w)\) for a weight \(\sigma\ge0\), and we let \(\Wp\) denote the even-parity energy space of \cite[\S2.1]{PS2020}, that is the space of
functions \(v\in L^2(\D)\) that are even with respect to \(\mu\), \(v(\cdot,-\mu)=v(\cdot,\mu)\), with
\(\mu\partial_zv\in L^2(\D)\). Such functions have traces,
and \(\langle v,v\rangle_{L^2_-}\) is finite, where the inflow form is
\[
\langle v,w\rangle_{L^2_-}
=\int_0^1v(0,\mu)w(0,\mu)\,\mu\dmu+\int_{-1}^0v(Z,\mu)w(Z,\mu)\,|\mu|\dmu.
\]
We write further \(\PP v(z)=\tfrac{1}{2}\int_{-1}^1v(z,\mu')\dmu'\) for the angular average,
and we identify functions of \(z\) alone with \(\mu\)-independent functions on \(\D\). 
We note that
\begin{equation}\label{eq:H1}
\|\mu\partial_zw\|^2_{L^2(\D)}=\tfrac23\|\partial_zw\|^2_{L^2(0,Z)}
\end{equation}
for any \(\mu\)-independent \(w\).
The \emph{diffusion subspace} \cite[(12)]{PS2020}
\[
\Wo=\{v\in\Wp:\ v=\PP v\}
\]
can thus be identified with \(H^1(0,Z)\), with equivalent norms.

\paragraph{Even-parity formulation} By \cite[Prob.~3.1]{PS2020}, the even part \(u\) of the
solution \(\psi\) of \eqref{eq:rte}--\eqref{eq:bc} is the solution of the problem to find
\(u\in\Wp\) such that
\begin{equation}\label{eq:ep}
a(u,v)=\ell(v)\qquad\text{for all }v\in\Wp,
\end{equation}
see also \cite{EggerSchlottbom12}, \cite[Rem.~3.4]{PS2020}.
Here \(a\) is the symmetric bilinear form
\[
a(v,w)=2\langle v,w\rangle_{L^2_-}
+\Big(\tfrac1{\sigma_t}\mu\partial_zv,\mu\partial_zw\Big)
+\big((\sigma_t-\sigma_s\PP)v,w\big),
\]
and \(\ell\in(\Wp)'\) denotes the linear form generated by \(q\) and \(g\) \cite[(8)]{PS2020}.
Since the odd part of \(\psi\) has vanishing angular average, we have that \(\PP u=\PP\psi=\phi\). For further
reference we also introduce the bilinear forms
\[
k(v,w)=(\sigma_s\PP v,w),\qquad b=a+k.
\]
Note that \(k(v,w)=(\sigma_s\PP v,\PP w)\), since \(\sigma_s\) does not depend on \(\mu\).
We assume:

\medskip
\noindent\textbf{(A1)} \(\sigma_s,\sigma_t\in L^\infty(0,Z)\) are non-negative and
\(\sigma_a=\sigma_t-\sigma_s\ge\gamma>0\) almost everywhere.\\
\medskip 

Under (A1) the form \(a\) is bounded, symmetric and coercive on \(\Wp\) \cite[Thm.~3.3]{PS2020},
and so is \(b\). The form \(b\) defines a transport energy
\begin{equation}\label{eq:energyb}
\|v\|_b^2=2\langle v,v\rangle_{L^2_-}+\|\mu\partial_zv\|^2_{1/\sigma_t}+\|v\|^2_{\sigma_t}.
\end{equation}

\paragraph{The DSA iteration in variational form} 
We use the even-parity formulation \eqref{eq:ep} to formulate and discretize the DSA iteration \eqref{eq:sweep}--\eqref{eq:update}.

\begin{algorithm_self}\label{alg:var}
Given \(\phi^{(n)}\in L^2(0,Z)\), find \(u^{(n+1/2)}\in\Wp\) and \(F^{(n+1)}\in\Wo\) such that
\begin{align}
b(u^{(n+1/2)},v)&=(\sigma_s\phi^{(n)},\PP v)+\ell(v) &&\text{for all }v\in\Wp,\label{eq:half}\\
a(F^{(n+1)},\chi)&=\big(\sigma_s(\phi^{(n+1/2)}-\phi^{(n)}),\chi\big) &&\text{for all }\chi\in\Wo,\label{eq:corr}
\end{align}
where \(\phi^{(n+1/2)}=\PP u^{(n+1/2)}\), and set \(\phi^{(n+1)}=\phi^{(n+1/2)}+F^{(n+1)}\).
\end{algorithm_self}

\begin{remark}\label{rem:equiv}
Algorithm~\ref{alg:var} coincides with \eqref{eq:sweep}--\eqref{eq:update} equipped with the
Marshak boundary conditions \(F^{(n+1)}\mp\tfrac{2}{3\sigma_t}\partial_zF^{(n+1)}=0\) at \(z=0\)
and \(z=Z\) \cite{EggerSchlottbom12,PS2020}, see also \cite{MorelMcGhee1995} for a corresponding
equivalence for the even-parity \(S_N\) equations.
 Indeed, since \(b\) contains no scattering, \eqref{eq:half} is the weak even-parity
form of the sweep \eqref{eq:sweep}, and \(\PP u^{(n+1/2)}=\phi^{(n+1/2)}\).
After integration by parts, \eqref{eq:corr} is the weak form of \eqref{eq:dsa} with the Marshak conditions as natural
boundary conditions. The boundary treatment of the diffusion problem is thus determined by the
formulation, whereas in other derivations it enters as a parameter on which the observed rates
depend noticeably \cite{PrinceEtAl2020}. 
\end{remark}

\paragraph{Discretization} The arguments of Sections~\ref{sec:it}--\ref{sec:main} use only the
forms \(a\), \(b\), \(k\) and a projection \(\LL\) onto the diffusion subspace, which is
introduced in Section~\ref{sec:main}. They carry over verbatim to the discrete iteration once
\(\LL\) maps the discrete space into its diffusion subspace. We therefore introduce the discrete
setting already here and treat both cases simultaneously. 
As in \cite{PS2020}, let \(V_h\subset H^1(0,Z)\) be a finite element space on an arbitrary mesh of \((0,Z)\)
and let \(Q_N\subset L^2(-1,1)\) be a finite dimensional space of even functions of \(\mu\) which contains the
constants.
We define the discrete spaces
\[
\Wph=V_h\otimes Q_N\subset\Wp,\qquad \Woh=\Wph\cap\Wo=V_h\otimes\mathbf 1 ,
\]
where \(\mathbf 1\) denotes the constant function \(1\) on \((-1,1)\). We identify \(\Woh\) with \(V_h\).
The discrete forms are the restrictions of \(a\), \(b\), \(k\) to \(\Wph\). The discrete DSA
iteration is Algorithm~\ref{alg:var} with \(\Wp\), \(\Wo\) replaced by \(\Wph\), \(\Woh\) and with
\(\phi^{(n)}\in\PP\Wph=V_h\), see also \cite[(18)--(20)]{PS2020}. Its error propagation operator
\(G_h\) is identified in Section~\ref{sec:it}, and the only additional property of the
discretization used below is that the weighted angular averages of Lemma~\ref{lem:lam} map
\(\Wph\) into \(\Woh\).

\section{Analysis of the DSA iteration}\label{sec:it}

For isotropic scattering the sweep enters the iteration only through the angular average of
its result. We therefore introduce the operators \(S\) and \(D\) on scalar functions that
correspond to the two steps \eqref{eq:half} and \eqref{eq:corr} without data. For
\(\eta\in L^2(0,Z)\), let \(u_\eta\in\Wp\) and \(D\eta\in\Wo\) be the solutions of
\begin{equation}\label{eq:SD}
b(u_\eta,v)=(\sigma_s\eta,\PP v)\quad\text{for all }v\in\Wp,\quad
a(D\eta,\chi)=(\sigma_s\eta,\chi)\quad\text{for all }\chi\in\Wo,
\end{equation}
and set \(S\eta=\PP u_\eta\). Thus \(S\) is a sweep with source \(\sigma_s\eta\) followed by
averaging, and \(D\) is the diffusion solve \eqref{eq:dsa}.
Since \(\sigma_s\) may vanish, the weighted products below are understood on the quotient of
\(L^2(0,Z)\) by the kernel of \((\sigma_s\cdot,\cdot)\), that is, on the weighted space
\(L^2(\sigma_s\dz)\).
Before proceeding, we collect some properties of \(S\) and \(D\).

\begin{lemma}\label{lem:SD}
\(S\) and \(D\) are self-adjoint and non-negative with respect to \((\sigma_s\cdot,\cdot)\), and
\(\|S\|_{\sigma_s}\le c\).
\end{lemma}

\begin{proof}
Testing the definition of \(u_\eta\) with \(u_\chi\) gives \((\sigma_s\chi,S\eta)=b(u_\chi,u_\eta)\),
hence \((\sigma_s\chi,S\eta)=(\sigma_s\eta,S\chi)\) and
\((\sigma_s\eta,S\eta)=\|u_\eta\|_b^2\ge0\). Similarly, \((\sigma_s\eta,D\chi)=a(D\eta,D\chi)\),
which equals \((\sigma_s\chi,D\eta)\). Finally, \(\sigma_s\le c\,\sigma_t\), \(S\eta=\PP u_\eta\) and
\(\|\PP u_\eta\|_{\sigma_t}\le\|u_\eta\|_{\sigma_t}\) together with \eqref{eq:energyb}
and the Cauchy--Schwarz inequality give
\[
\|S\eta\|^2_{\sigma_s}\le c\,\|\PP u_\eta\|^2_{\sigma_t}\le c\,\|u_\eta\|^2_b
=c\,(\sigma_s\eta,S\eta)\le c\,\|\eta\|_{\sigma_s}\|S\eta\|_{\sigma_s},
\]
which proves \(\|S\|_{\sigma_s}\le c\).
\end{proof}

With \(u_\ell\in\Wp\) defined by
\(b(u_\ell,v)=\ell(v)\) for all \(v\in\Wp\) and \(f=\PP u_\ell\), the steps of
Algorithm~\ref{alg:var} read
\[
\phi^{(n+1/2)}=S\phi^{(n)}+f,\qquad F^{(n+1)}=D\big(\phi^{(n+1/2)}-\phi^{(n)}\big).
\]
For the full step we thus obtain
\[
\phi^{(n+1)}=(I-(I+D)(I-S))\phi^{(n)}+(I+D)f.
\]
Since \(a=b-k\), the solution \(u\) of \eqref{eq:ep} satisfies
\(b(u,v)=(\sigma_s\phi,\PP v)+\ell(v)\) for all \(v\in\Wp\), where \(\phi=\PP u\), that is,
\((I-S)\phi=f\). We thus recognize the well-known fact that the DSA iteration is a Richardson iteration for the scalar flux equation
with preconditioner \(I+D\) \cite{AdamsLarsen02}.
The error \(e^{(n)}=\phi-\phi^{(n)}\) satisfies the linear iteration
\begin{equation}\label{eq:G}
e^{(n+1)}=G\,e^{(n)}\qquad\text{with}\qquad G=I-(I+D)(I-S).
\end{equation}
The same holds for the discrete iteration, with \(G_h=I-(I+D_h)(I-S_h)\), where \(S_h\) and \(D_h\)
are defined by \eqref{eq:SD} with \(\Wp\), \(\Wo\) replaced by \(\Wph\), \(\Woh\).

Without the correction step, \(e^{(n+1)}=Se^{(n)}\), and Lemma~\ref{lem:SD} gives the rate \(c\),
see also \cite{DPS2022,PS2020}.
To obtain improved bounds, the correction has to enter quantitatively.

\begin{lemma}\label{lem:selfadj}
\(G\) is self-adjoint with respect to \(\langle x,y\rangle_\star=(\sigma_s(I+D)^{-1}x,y)\), and
\[
\|G\|_\star=\rho(G)=\max\big\{|1-\lambda|:\ \lambda\in\sigma\big((I+D)(I-S)\big)\big\}.
\]
\end{lemma}

\begin{proof}
\((I+D)^{-1}G=(I+D)^{-1}-(I-S)\) is a difference of two operators which are self-adjoint with
respect to \((\sigma_s\cdot,\cdot)\) by Lemma~\ref{lem:SD}. For a self-adjoint operator, norm
and spectral radius agree.
\end{proof}

To compare the sweep with the diffusion correction, we introduce the
energies
\begin{gather*}
X(f)=\sup_{v\in\Wp}\big(2(f,\PP v)-b(v,v)\big),\qquad
Y(f)=\sup_{\chi\in\Wo}\big(2(f,\PP\chi)-b(\chi,\chi)\big),\\
m(f)=(f,f/\sigma_s),
\end{gather*}
for \(f\in L^2(0,Z)\),
with \(m(f)=\infty\) unless \(f\) vanishes where \(\sigma_s\) does. \(X(f)\) is the energy of the
transport problem without scattering with the isotropic source \(f\), and \(Y(f)\) is the same
supremum restricted to the diffusion subspace.
Since \(\Wo\subset\Wp\), \(Y\le X\).

\begin{lemma}\label{lem:XY}
For \(f=\sigma_s\eta\) with \(\eta\in L^2(0,Z)\),
\[
m(f)-X(f)=\big(\sigma_s(I-S)\eta,\eta\big),\qquad m(f)-Y(f)=\big(\sigma_s(I+D)^{-1}\eta,\eta\big).
\]
\end{lemma}

\begin{proof}
The supremum defining \(X(f)\) is attained at the solution \(v\in\Wp\) of
\(b(v,w)=(\sigma_s\eta,\PP w)\) for all \(w\in\Wp\), that is at \(u_\eta\), and
\(X(f)=b(u_\eta,u_\eta)=(\sigma_s\eta,S\eta)\). The supremum defining \(Y(f)\) is attained at
\(\chi_\eta\in\Wo\) with \(b(\chi_\eta,\chi)=(\sigma_s\eta,\chi)\) for all \(\chi\in\Wo\). Since
\(b=a+k\) and \(k(\chi_\eta,\chi)=(\sigma_s\chi_\eta,\chi)\) on \(\Wo\), this means
\(a(\chi_\eta,\chi)=\big(\sigma_s(\eta-\chi_\eta),\chi\big)\), i.e.,
\(\chi_\eta=D(\eta-\chi_\eta)\) and \(\chi_\eta=(I+D)^{-1}D\eta\). Hence
\(Y(f)=(\sigma_s\eta,\chi_\eta)=(\sigma_s\eta,(I+D)^{-1}D\eta)\). Since \(m(f)=(\sigma_s\eta,\eta)\)
and \(I-(I+D)^{-1}D=(I+D)^{-1}\), the claim follows.
\end{proof}

\begin{lemma}\label{lem:lam1}
We have \(I+D\le(I-S)^{-1}\) with respect to \((\sigma_s\cdot,\cdot)\) and
\(\sigma\big((I+D)(I-S)\big)\subset[1-c,1]\). Moreover, the smallest spectral value
\(\lambda_{\min}\) of \((I+D)(I-S)\) satisfies
\[
1-\rho(G)=\lambda_{\min}=\inf_{f}\frac{m(f)-X(f)}{m(f)-Y(f)} ,
\]
where the infimum is taken over \(f=\sigma_s\eta\ne0\), \(\eta\in L^2(0,Z)\).
\end{lemma}
\begin{proof}
By Lemma~\ref{lem:XY} and \(Y\le X\), \(\big(\sigma_s(I-S)\eta,\eta\big)\le
\big(\sigma_s(I+D)^{-1}\eta,\eta\big)\) for all \(\eta\), that is, \(I-S\le(I+D)^{-1}\). Both
operators are positive and self-adjoint with respect to \((\sigma_s\cdot,\cdot)\) by
Lemma~\ref{lem:SD}, so this is equivalent to \(I+D\le(I-S)^{-1}\). Moreover, \((I+D)(I-S)\) is
self-adjoint with respect to \(\langle\cdot,\cdot\rangle_\star\) by Lemma~\ref{lem:selfadj}. We thus have that
the  Rayleigh quotients
\[
\frac{\langle(I+D)(I-S)\eta,\eta\rangle_\star}{\langle\eta,\eta\rangle_\star}
=\frac{\big(\sigma_s(I-S)\eta,\eta\big)}{\big(\sigma_s(I+D)^{-1}\eta,\eta\big)}
\]
are bounded from above by \(1\). Since \(I\leq I+D\), we have 
\((\sigma_s(I+D)^{-1}\eta,\eta)\leq \|\eta\|_{\sigma_s}^2\), which yields the lower bound \(1-c\) by Lemma~\ref{lem:SD}. Since \((I+D)(I-S)\) is
self-adjoint with respect to \(\langle\cdot,\cdot\rangle_\star\), its spectrum lies in the closure
of the set of Rayleigh quotients, and \(\lambda_{\min}\) is their infimum. By
Lemma~\ref{lem:XY}, the Rayleigh quotient at \(\eta\) equals \((m(f)-X(f))/(m(f)-Y(f))\) with
\(f=\sigma_s\eta\). Finally, \(\rho(G)=\max\{|1-\lambda|:\lambda\in\sigma((I+D)(I-S))\}
=1-\lambda_{\min}\) by Lemma~\ref{lem:selfadj} and \(\sigma((I+D)(I-S))\subset[1-c,1]\).
\end{proof}

\subsection{A primal criterion for the rate}\label{sec:scalar}

By Lemma~\ref{lem:lam1}, a lower bound for the quotient \((m-X)/(m-Y)\) yields an upper bound
for \(\rho(G)\). The functionals \(X\) and \(Y\) are the Fenchel conjugates of the quadratic forms \(b|_{\Wp}\) and
\(b|_{\Wo}\), evaluated at the functional \((f,\PP\,\cdot)\). Writing \(Y\) as a supremum over
\(\chi\) turns the min--max form into a primal condition, which is what we shall verify.

\begin{lemma}\label{lem:primal}
Let \(0<r\le1\). If for every \(v\in\Wp\) there is \(\chi\in\Wo\) with
\begin{equation}\label{eq:fenchel}
\big\|\PP v-(1-r)\chi\big\|^2_{\sigma_s}+r(1-r)\,b(\chi,\chi)\ \le\ r\,b(v,v),
\end{equation}
then \(\rho(G)\le r\).
\end{lemma}

\begin{proof}
Let \(f=\sigma_s\eta\ne0\), so that \(m(f)=\|\eta\|^2_{\sigma_s}<\infty\). Let \(v\in\Wp\), let
\(\chi\in\Wo\) be as in \eqref{eq:fenchel}, and set \(g=\PP v-(1-r)\chi\). Since \(\PP\chi=\chi\),
\[
2(f,\PP v)=2(f,g)+2(1-r)(f,\PP\chi).
\]
By the Cauchy--Schwarz and Young inequalities,
\[
2(f,g)=2(\sigma_s\eta,g)\le2\|\eta\|_{\sigma_s}\|g\|_{\sigma_s}
\le r\,m(f)+r^{-1}\|g\|^2_{\sigma_s},
\]
and \eqref{eq:fenchel} gives \(r^{-1}\|g\|^2_{\sigma_s}\le b(v,v)-(1-r)\,b(\chi,\chi)\).
Combining these estimates and using the definition of \(Y\), we obtain
\[
2(f,\PP v)-b(v,v)\le r\,m(f)+(1-r)\big(2(f,\PP\chi)-b(\chi,\chi)\big)\le r\,m(f)+(1-r)Y(f).
\]
Taking the supremum over \(v\in\Wp\) gives \(X(f)\le r\,m(f)+(1-r)Y(f)\), that is,
\[
m(f)-X(f)\ge(1-r)\big(m(f)-Y(f)\big).
\]
By Lemma~\ref{lem:XY}, \(m(f)-Y(f)=(\sigma_s(I+D)^{-1}\eta,\eta)>0\), so that the quotient in
Lemma~\ref{lem:lam1} is at least \(1-r\) at this \(f\). Since \(\eta\) was arbitrary, the
infimum satisfies \(\lambda_{\min}\ge1-r\), and \(\rho(G)=1-\lambda_{\min}\le r\).
\end{proof}

\begin{remark}\label{rem:necessary}
Condition \eqref{eq:fenchel} is also necessary for \(0<r<1\):
if it fails, then $\rho(G)>r$. 
In fact, let \(\alpha=1-r\), and let
\(v\in\Wp\) be such that \eqref{eq:fenchel} fails for every \(\chi\in\Wo\). Its left-hand side
is a strictly convex and coercive quadratic in \(\chi\), hence minimal at the unique
\(\chi_v\in\Wo\) with
\begin{equation}\label{eq:chiv}
r\,b(\chi_v,\xi)+\alpha(\sigma_s\chi_v,\xi)=(\sigma_s\PP v,\xi)
\qquad\text{for all }\xi\in\Wo.
\end{equation}
Set \(g=\PP v-\alpha\chi_v\) and \(f=\sigma_sg/r\). Testing \eqref{eq:chiv} with
\(\xi=\chi_v\) and using \(g+\alpha\chi_v=\PP v\), the minimal value of the left-hand side of
\eqref{eq:fenchel} is \(\|g\|^2_{\sigma_s}+\alpha(\sigma_sg,\chi_v)=(\sigma_sg,\PP v)\).
Failure of \eqref{eq:fenchel} therefore means \((f,\PP v)>b(v,v)\); in particular \(f\ne0\).

In terms of \(f\), \eqref{eq:chiv} reads \(b(\chi_v,\xi)=(f,\PP\xi)\) for all \(\xi\in\Wo\),
which is the equation characterizing the maximizer in the supremum defining \(Y\), see the
proof of Lemma~\ref{lem:XY}. Thus \(\chi_v\) attains that supremum and \(Y(f)=(f,\chi_v)\).
Since \(r\,m(f)=(f,g)\) and \(\PP v=g+\alpha\chi_v\), we find
\((f,\PP v)=r\,m(f)+\alpha\,Y(f)\). Testing the supremum defining \(X(f)\) with \(v\),
\[
X(f)-\alpha Y(f)-r\,m(f)\ \ge\ 2(f,\PP v)-b(v,v)-\alpha Y(f)-r\,m(f)=(f,\PP v)-b(v,v)>0 .
\]
Hence \(m(f)-X(f)<\alpha\big(m(f)-Y(f)\big)\), and \(m(f)-Y(f)>0\) by Lemma~\ref{lem:XY}, so
that the infimum in Lemma~\ref{lem:lam1} is smaller than \(\alpha\) and \(\rho(G)>r\). 
\end{remark}

Let \(S_h\), \(D_h\), \(G_h\), \(X_h\) and \(Y_h\) be defined as above with \(\Wp\), \(\Wo\) replaced by
\(\Wph\), \(\Woh\), and with \(\eta\) ranging over \(\PP\Wph=V_h\).
While \(f=\sigma_s\eta\) is not necessarily in \(V_h\),
\(X_h(f)\) and \(Y_h(f)\) are still well-defined.

\begin{proposition}\label{prop:discred}
The error representation \eqref{eq:G} and Lemmas~\ref{lem:SD} to \ref{lem:primal} hold
verbatim for the discrete iteration, with \(X\), \(Y\) replaced by \(X_h\), \(Y_h\). In
particular \(\rho(G_h)=\|G_h\|_{\star,h}\), and \(\rho(G_h)\le r\) if for every \(v\in\Wph\) there is
\(\chi\in\Woh\) with \eqref{eq:fenchel}.
\end{proposition}

\begin{proof}
The proofs use only the forms, the inclusion of the diffusion subspace in the transport space,
and the fact that \(\PP\) maps the diffusion subspace onto itself,
which holds for
\(\Woh\subset\Wph\). In Lemmas~\ref{lem:XY} and~\ref{lem:primal}, \(f=\sigma_s\eta\) with
\(\eta\in\PP\Wph=V_h\); the suprema defining \(X_h(f)\) and \(Y_h(f)\) are then attained at the
discrete solutions, which gives the identities of Lemma~\ref{lem:XY} with \(S_h\) and \(D_h\) in
place of \(S\) and \(D\).
\end{proof}

Note that \(X_h\) and \(Y_h\) are not the restrictions of \(X\) and \(Y\) to discrete arguments.
Since the suprema run over smaller spaces, both energies decrease. Because \(X\) and \(Y\) enter
the quotient of Lemma~\ref{lem:lam1}, no comparison between
\(\rho(G_h)\) and \(\rho(G)\) follows from the Courant--Fischer principle. What survives
discretization is the criterion \eqref{eq:fenchel} itself, provided \(\chi\) is constructed from
\(v\) by a map that preserves the discrete spaces.

\section{The main result}\label{sec:main}

For \(v\in\Wp\) the angular average \(\PP v\) belongs to
\(H^{1/2}(0,Z)\) by the averaging lemma \cite{GolseLionsPerthameSentis1988}, but in general not
to \(\Wo\cong H^1(0,Z)\). Therefore \(\PP\) cannot serve as an interpolant into the diffusion
subspace. Instead, for a weight \(w\) on \((-1,1)\) we consider the weighted angular average
\(\LL_wv=\PP(wv)\).
We always assume that \(w\ge0\) is even, \(w\not\equiv0\) and \(w/\mu\in L^2(-1,1)\).
The weight is a device of the analysis. Algorithm~\ref{alg:var} and its discretization are not
modified, and neither \(w\) nor the parameters chosen for it below enter the iteration.
We set
\begin{equation}\label{eq:sharpconst}
C_\partial(w)=\PP(|\mu|)\,\PP\Big(\frac{w^2}{|\mu|}\Big),\qquad
C_T(w)=\PP(\mu^2)\,\PP\Big(\frac{w^2}{\mu^2}\Big),\qquad
C_A(w)=\PP(w^2).
\end{equation}
Each of them is of the form \(\PP(\varphi)\,\PP(w^2/\varphi)\) and belongs to one term of the
energy \eqref{eq:energyb}, namely \(\varphi=|\mu|\) to the boundary term, \(\varphi=\mu^2\) to
the transport term and \(\varphi=\mathbf 1\) to the attenuation term. Here \(\varphi\) is the
weight in \(\mu\) that the term carries, and the constant comes from the Cauchy--Schwarz
inequality applied to \(wv=(w/\sqrt\varphi)\,(\sqrt\varphi\,v)\). 

\begin{lemma}\label{lem:lam}
\(\LL_w\) maps \(\Wp\) into \(\Wo\) and \(\Wph\) into \(\Woh\), and for \(v\in\Wp\),
\begin{enumerate}
\item[(i)] \(\langle\LL_wv,\LL_wv\rangle_{L^2_-}\le C_\partial(w)\,\langle v,v\rangle_{L^2_-}\);
\item[(ii)] \(\|\mu\partial_z\LL_wv\|^2_{1/\sigma_t}\le C_T(w)\,\|\mu\partial_zv\|^2_{1/\sigma_t}\);
\item[(iii)] \(\|\LL_wv\|^2_{\sigma_t}\le C_A(w)\,\|v\|^2_{\sigma_t}\).
\end{enumerate}
If in addition \(\PP(w)=1\), then \(\LL_w\) is a projection of \(\Wp\) onto \(\Wo\).
\end{lemma}

\begin{proof}
Since \(\LL_w\) acts only in \(\mu\) it commutes with \(\partial_z\), and
\(\partial_z\LL_wv=\PP\big(\tfrac w\mu\cdot\mu\partial_zv\big)\), so that the
Cauchy--Schwarz inequality gives
\begin{equation}\label{eq:ptwise}
|\partial_z\LL_wv|^2\le\PP(w^2/\mu^2)\,\PP\big(|\mu\partial_zv|^2\big)
\qquad\text{pointwise in }z .
\end{equation}
Since \(w/\mu\in L^2(-1,1)\), the factor \(\PP(w^2/\mu^2)\) is finite, so that
\(\LL_wv\in\Wo\) by \eqref{eq:H1}. For \(v=\sum_iv_i\otimes q_i\) with \(v_i\in V_h\) and
\(q_i\in Q_N\) one has \(\LL_wv=\sum_i\PP(wq_i)\,v_i\in V_h\cong\Woh\). If \(\PP(w)=1\), then
\(\LL_w\mathbf 1=\mathbf 1\).

(i) Since \(v\) and \(w\) are even, \(\LL_wv(0)=\int_0^1w(\mu)v(0,\mu)\dmu\), and the
Cauchy--Schwarz inequality with the weight \(\mu\) gives
\begin{align*}
|\LL_wv(0)|^2=\Big|\int_0^1\frac{w}{\sqrt\mu}\,\sqrt\mu\,v(0,\mu)\dmu\Big|^2
\leq \PP(w^2/|\mu|)\int_0^1\mu\,v(0,\mu)^2\dmu,
\end{align*}
and similarly at \(z=Z\). As \(\LL_wv\) is independent of \(\mu\), the left hand side of~(i) equals
\(\PP(|\mu|)\big(|\LL_wv(0)|^2+|\LL_wv(Z)|^2\big)\), and the two bounds give (i).

(ii) For \(\mu\)-independent \(\eta\) one has
\(\|\mu\partial_z\eta\|^2_{1/\sigma_t}=\tfrac23\int_0^Z\sigma_t^{-1}|\partial_z\eta|^2\dz\) by
\eqref{eq:H1}, while \(\|\mu\partial_zv\|^2_{1/\sigma_t}=2\int_0^Z\sigma_t^{-1}
\PP\big((\mu\partial_zv)^2\big)\dz\). Thus, by \eqref{eq:ptwise},
\begin{align*}
\|\mu\partial_z\LL_wv\|^2_{1/\sigma_t}
&=2\PP(\mu^2)\int_0^Z\frac{|\partial_z\LL_wv|^2}{\sigma_t}\dz
\le2\PP(\mu^2)\PP(w^2/\mu^2)\int_0^Z\frac{\PP\big((\mu\partial_zv)^2\big)}{\sigma_t}\dz\\
&=C_T(w)\,\|\mu\partial_zv\|^2_{1/\sigma_t}.
\end{align*}

(iii) We have \(|\LL_wv|^2=|\PP(wv)|^2\le C_A(w)\,\PP(v^2)\) pointwise
in \(z\). Multiplying by \(2\sigma_t\) and integrating in \(z\) gives (iii), because
\(\|\eta\|^2_{\sigma_t}=2\int_0^Z\sigma_t\PP(\eta^2)\dz\).
\end{proof}

\subsection{Conditions for the weight}\label{sec:cond}

Throughout this subsection, \(0<r<1\) and \(\alpha=1-r\). On \(L^2(-1,1)\) we consider the collision
form
\begin{equation}\label{eq:collform}
\QQ(f,g)=c\,\theta(f)\,\theta(g)+r\alpha\,\PP(wf)\,\PP(wg)-r\,\PP(fg),
\,\,\,\,
\theta(g)=\PP g-\alpha\PP(wg),
\end{equation}
and the matrices
\[
N_w=c\begin{pmatrix}1&-\alpha\\-\alpha&\alpha^2\end{pmatrix}
+r\alpha\begin{pmatrix}0&0\\0&1\end{pmatrix} ,
\qquad
\Gamma_w=\begin{pmatrix}1&\PP(w)\\ \PP(w)&\PP(w^2)\end{pmatrix} .
\]
Here \(\Gamma_w\) is the Gram matrix of \(\{\mathbf 1,w\}\) with respect to \(\PP(\cdot\,\cdot)\). 
Since \(w/\mu\in L^2(-1,1)\) and \(w\neq 0\), \(w\) and \(\mathbf 1\) are linearly independent and 
\(\Gamma_w\) is positive definite.

\begin{lemma}\label{lem:coll}
The following three statements are equivalent:
\begin{enumerate}
\item[(i)] \(\QQ(g,g)\le0\) for all \(g\in L^2(-1,1)\), that is,
\begin{equation}\label{eq:collw}
c\,\big|\PP g-\alpha\PP(wg)\big|^2+r\alpha\big|\PP(wg)\big|^2\ \le\ r\,\PP(g^2)
\qquad\text{for all }g\in L^2(-1,1);
\end{equation}
\item[(ii)] \(\QQ(g,g)\le0\) for all \(g\) in the two dimensional space
\(\mathcal V=\operatorname{span}\{\mathbf 1,w\}\);
\item[(iii)] \(\lambda_{\max}(N_w\Gamma_w)\le r\).
\end{enumerate}
\end{lemma}

\begin{proof}
(i)\(\Rightarrow\)(ii) is trivial. For (ii)\(\Rightarrow\)(i), decompose \(g=g_0+h\) with
\(g_0\in\mathcal V\) and \(h\) orthogonal to \(\mathcal V\) with respect to \(\PP(\cdot\,\cdot)\). Then
\(\PP h=\PP(wh)=0\), hence \(\theta(h)=0\) and \(\PP(g_0h)=0\), so that
\(\QQ(g,g)=\QQ(g_0,g_0)-r\,\PP(h^2)\le\QQ(g_0,g_0)\).

For (ii)\(\Leftrightarrow\)(iii) we use the basis \(\{\mathbf 1,w\}\) of \(\mathcal V\).
For \(g=y_1+y_2w\) and
\(y=(y_1,y_2)^\top\),
\[
x=\begin{pmatrix}\PP g\\ \PP(wg)\end{pmatrix}=\Gamma_wy,\qquad \PP(g^2)=y^\top\Gamma_wy,
\]
so that
\[
\QQ(g,g)=x^\top N_wx-r\,y^\top\Gamma_wy=y^\top\big(\Gamma_wN_w\Gamma_w-r\Gamma_w\big)y .
\]
Hence (ii) holds if and only if \(\Gamma_wN_w\Gamma_w\le r\,\Gamma_w\), which
 is equivalent to \(\Gamma_w^{1/2}N_w\Gamma_w^{1/2}\le rI\),
that is, to (iii), because \(\Gamma_w^{1/2}N_w\Gamma_w^{1/2}\) and \(N_w\Gamma_w\) have the same
eigenvalues.
\end{proof}

\begin{theorem}\label{thm:halfw}
Let \(\alpha=1-r\) for \(r\in(0,1)\) and assume
\begin{equation}\label{eq:H}
\alpha\,C_\partial(w)\le 1,\qquad
\alpha\,C_T(w)\le1 ,\qquad
\QQ(g,g)\le0\quad\text{for all }g\in L^2(-1,1),
\end{equation}
with \(\QQ\) as in \eqref{eq:collform}. Then \(\rho(G)\le r\) and \(\rho(G_h)\le r\).
\end{theorem}

\begin{proof}
Let \(v\in\Wp\) and \(\chi=\LL_wv\in\Wo\), which is admissible by Lemma~\ref{lem:lam}. By
\eqref{eq:energyb}, the criterion \eqref{eq:fenchel} reads
\begin{multline*}
\|\PP v-\alpha\chi\|^2_{\sigma_s}
+r\alpha\Big(2\langle\chi,\chi\rangle_{L^2_-}+\|\mu\partial_z\chi\|^2_{1/\sigma_t}
+\|\chi\|^2_{\sigma_t}\Big)\\
\le r\Big(2\langle v,v\rangle_{L^2_-}+\|\mu\partial_zv\|^2_{1/\sigma_t}+\|v\|^2_{\sigma_t}\Big),
\end{multline*}
and it suffices to verify the three inequalities
\begin{gather*}
\alpha\langle\chi,\chi\rangle_{L^2_-}\le\langle v,v\rangle_{L^2_-},\qquad
\alpha\|\mu\partial_z\chi\|^2_{1/\sigma_t}\le\|\mu\partial_zv\|^2_{1/\sigma_t},\\
\|\PP v-\alpha\chi\|^2_{\sigma_s}+r\alpha\|\chi\|^2_{\sigma_t}\le r\|v\|^2_{\sigma_t},
\end{gather*}
since multiplying the first two by \(r\) and adding the third gives the displayed inequality.
The first two inequalities are Lemma~\ref{lem:lam}(i),(ii) combined with the
\(C_\partial\)- and the \(C_T\)-condition in \eqref{eq:H}.
For the third inequality, all three terms are integrals over \(z\) of quantities that depend only on
\(g=v(z,\cdot)\in L^2(-1,1)\), for almost every \(z\). Indeed \(\PP v(z)=\PP g\),
\(\chi(z)=\PP(wg)\), and \(\|\eta\|^2_\sigma=2\int_0^Z\sigma|\eta|^2\dz\) for \(\mu\)-independent
\(\eta\), while \(\|v\|^2_{\sigma_t}=2\int_0^Z\sigma_t\PP(g^2)\dz\). Since \(\sigma_s\le c\,\sigma_t\),
it therefore suffices to show \eqref{eq:collw}, which is the \(\QQ\)-condition in
\eqref{eq:H}.

By Lemma~\ref{lem:primal}, \(\rho(G)\le r\). For the discrete iteration nothing changes.
 The first two inequalities hold for every \(v\in\Wp\), hence for \(v\in\Wph\).
 Condition \eqref{eq:collw} is a pointwise statement
in \(z\) for arbitrary \(g\), and \(\chi=\LL_wv\in\Woh\) by Lemma~\ref{lem:lam}. Hence
\(\rho(G_h)\le r\) by Proposition~\ref{prop:discred}.
\end{proof}

\begin{remark}\label{rem:half}
The simplest admissible weight already gives a good rate, and all quantities of
Theorem~\ref{thm:halfw} are then explicit. For \(w=2\beta|\mu|\) with \(\beta=3/(4-c)\) and
\(r=c/4\) we have \(\alpha=(4-c)/4\) and \(\alpha\beta=\tfrac34\). 
The moments in \eqref{eq:sharpconst} are
\begin{align*}
\PP(w)=\beta,\quad \PP(w^2)=\tfrac43\beta^2,\quad
\PP(w^2/|\mu|)=2\beta^2,\quad \PP(w^2/\mu^2)=4\beta^2,
\end{align*}
so that \(C_\partial(w)=\beta^2\) and \(C_A(w)=C_T(w)=\tfrac43\beta^2\), while
\(\Gamma_w=\left(\begin{smallmatrix}1&\beta\\\beta&\tfrac43\beta^2\end{smallmatrix}\right)\).
Hence,
\(\alpha\,C_\partial(w)\leq \frac34\), and 
\(\alpha\,C_T(w)\le1\).
The eigenvalues of \(N_w\Gamma_w\) are \(c/4\) and \(3c/(4(4-c))\), both at most \(c/4\), so that
\eqref{eq:H} holds by Lemma~\ref{lem:coll}. Theorem~\ref{thm:halfw} therefore gives
\[
\rho(G)\le\frac c4\qquad\text{and}\qquad\rho(G_h)\le\frac c4 ,
\]
with the interpolant \(\chi=\beta\LL_{2|\mu|}v\). 
\end{remark}
Remark~\ref{rem:half} shows that  \(c/4\) is a robust bound obtained from a simple universal weight,
 while the remainder of Section~\ref{sec:main} constructs a \(c\)-dependent weight that recovers the sharper infinite-medium rate.

\subsection{A weight yielding \texorpdfstring{\(\rho_\infty(c)\)}{rho\_infty(c)}}\label{sec:sharp:opt}
We now choose \(w\) so that the three conditions of Theorem~\ref{thm:halfw} are satisfied with \(r=\rho_\infty(c)\).
Let us introduce the functions
\begin{equation}\label{eq:sfdef}
A(\lambda)=\frac{\arctan\lambda}{\lambda},\quad
B(\lambda)=\frac1{1+\lambda^2},\quad
Y_\infty(\lambda)=\frac{3}{3+\lambda^2},
\end{equation}
where here and throughout this subsection, we assume \(\lambda>0\).
\(A\) is the symbol of the transport sweep and \(Y_\infty\) the
symbol of the diffusion correction for the infinite homogeneous medium, and a short
computation rewrites the symbol \eqref{eq:alsymbol} of \cite{AdamsLarsen02} as
\begin{equation}\label{eq:omdef}
\omega_c=\frac{A-Y_\infty}{1/c-Y_\infty}=\frac{c\,(A-Y_\infty)}{1-c\,Y_\infty},
\qquad\text{so that}\qquad
\rho_\infty(c)=\sup_{\lambda>0}\omega_c(\lambda).
\end{equation}
We state some basic properties of \(A\), \(B\), \(Y_\infty\) and \(\omega_c\) for later reference.
\begin{lemma}\label{lem:sym}
We have
\begin{enumerate}
\item[(i)] \(0<B<A<1\) as well as \(A>Y_\infty\).
\item[(ii)] \(\omega_c>0\) on \((0,\infty)\) and \(\omega_c(\lambda)\to0\) as \(\lambda\to\infty\),
for every \(c\in(0,1]\).
\end{enumerate}
\end{lemma}
\begin{proof}
(i) Clearly \(B>0\) and \(A<1\). The functions \(\lambda A\), \(\lambda B\) and \(\lambda Y_\infty\)
vanish at \(\lambda=0\), and (i) follows from \((\lambda A)'>(\lambda B)'\) and
\((\lambda A)'>(\lambda Y_\infty)'\).

(ii) Using (i) and \(cY_\infty<1\) shows \(\omega_c>0\). Both \(A\) and \(Y_\infty\) tend to \(0\) as
\(\lambda\to\infty\), and so does \(\omega_c\).
\end{proof}
For an isotropic plane wave of frequency \(\lambda\), the infinite-domain sweep produces and angular profile with even part \(g_\lambda\), whose average is the scalar-flux symbol \(A\) in \eqref{eq:sfdef}.
For parameters $\lambda>0$ and $\kappa>0$, which we fix below, we take the weight proportional to the complementary profile,
\begin{equation}\label{eq:wlam}
w_\lambda=\kappa\,(1-g_\lambda)=\kappa\,\frac{\lambda^2\mu^2}{1+\lambda^2\mu^2},
\,\,\,\,\text{with } g_\lambda(\mu)=\frac1{1+\lambda^2\mu^2}.
\end{equation}
Each \(w_\lambda\) is even, nonnegative and satisfies \(w_\lambda/\mu\in L^\infty(-1,1)\), so the previous results apply.
The following identities follow by direct integration.

\begin{lemma}\label{lem:mom}
For every \(\lambda>0\), with \(A=A(\lambda)\) and \(B=B(\lambda)\),
\begin{enumerate}
    \item[(i)] \(\PP(g_\lambda)=A\), \(\PP(g_\lambda^2)=\tfrac12(A+B)\).
    \item[(ii)] \(\PP(1-g_\lambda)=1-A\), \(\PP\big((1-g_\lambda)^2\big)=1-\tfrac12(3A-B)\), \(\PP\big(g_\lambda(1-g_\lambda)\big)=\tfrac12(A-B)\)
    \item[(iii)] \(\PP\big((1-g_\lambda)^2/\mu^2\big)=\tfrac12\lambda^2(A-B)\)
    \item[(iv)] \(\PP\big((1-g_\lambda)^2/|\mu|\big)=\tfrac12\big(\ln(1+\lambda^2)+B-1\big)\).
\end{enumerate}
\end{lemma}

Our aim is to apply Theorem~\ref{thm:halfw} with \(r=\rho_\infty(c)\), which requires bounds
for \(C_\partial(w_\lambda)\) and \(C_T(w_\lambda)\).
For \(w=w_\lambda\), Lemma~\ref{lem:mom} yields
\begin{align}
    \label{eq:momw2}
    C_\partial(w_\lambda)&=\tfrac14\kappa^2\big(\ln(1+\lambda^2)+B-1\big),\quad
    C_T(w_\lambda)=\tfrac16\kappa^2\lambda^2(A-B),
    \\\label{eq:momw}
    C_A(w_\lambda)&=\kappa^2\big(1-\tfrac12(3A-B)\big),\,\,\,
    \PP(w_\lambda)=\kappa(1-A),\,\,\,
    \PP(w_\lambda g_\lambda)=\tfrac12\kappa(A-B).
\end{align}
The verification of the \(\QQ\)-condition in \eqref{eq:H} will be done by
investigating the sign of
\begin{equation}\label{eq:sfh}
h(\lambda)=(A-B)(3+\lambda^2)-6(1-A).
\end{equation}
Comparing the constants of Lemma~\ref{lem:lam} amounts to comparing the moments
\(\PP(w_\lambda^2/\varphi)\) of one and the same function \(w_\lambda^2\). 
The transport constant \(C_T\) will be the limiting constant in the sharp construction.
We therefore normalize the angular measure associated with \(w_\lambda^2/\mu^2\), and set
\begin{equation}\label{eq:nudef}
\nu_\lambda=\frac{w_\lambda^2/\mu^2}{\PP\big(w_\lambda^2/\mu^2\big)}
=\frac{2\lambda^2\mu^2g_\lambda^2}{A-B},
\qquad M_\lambda(\varphi)=\PP(\nu_\lambda\varphi),
\end{equation}
where we used Lemma~\ref{lem:mom}(iii).
Clearly, \(\nu_\lambda\ge0\) is even and \(M_\lambda(\mathbf 1)=1\). The next lemma expresses the
comparisons with \(C_T(w_\lambda)\) and the sign of \(h\) as averages against \(\nu_\lambda\), which
are monotone in \(\lambda\).

\begin{lemma}\label{lem:nu}
We have
\begin{equation}\label{eq:CTratio}
\frac{C_\partial(w_\lambda)}{C_T(w_\lambda)}=\frac{\PP(|\mu|)}{\PP(\mu^2)}\,M_\lambda(|\mu|),
\qquad
\frac{C_A(w_\lambda)}{C_T(w_\lambda)}=\frac{M_\lambda(\mu^2)}{\PP(\mu^2)} ,
\end{equation}
and
\begin{equation}\label{eq:hM}
h=3\lambda^2(A-B)\big(\PP(\mu^2)-M_\lambda(\mu^2)\big) .
\end{equation}
Moreover, \(\lambda\mapsto M_\lambda(\varphi)\) is strictly decreasing on \((0,\infty)\) for every
even \(\varphi\) that is strictly increasing on \((0,1)\).
\end{lemma}
\begin{proof}
By \eqref{eq:nudef}, \(M_\lambda(\mu^2/\varphi)=\PP(w_\lambda^2/\varphi)/\PP(w_\lambda^2/\mu^2)\)
for \(\varphi=\mathbf 1\) and \(\varphi=|\mu|\), which gives \eqref{eq:CTratio} by
\eqref{eq:sharpconst}. Lemma~\ref{lem:mom}(ii),(iii) gives
\(\lambda^2(A-B)\,M_\lambda(\mu^2)=2-3A+B\), and \(A>B\) by Lemma~\ref{lem:sym}(i); inserting
this into \eqref{eq:hM} and using \(\PP(\mu^2)=\tfrac13\) yields \eqref{eq:sfh}.

Let \(0<\lambda<\tilde\lambda\). By \eqref{eq:nudef}, \(\nu_{\tilde\lambda}=\rho\,\nu_\lambda\)
with \(\rho\) a positive multiple of \(\big((1+\lambda^2\mu^2)/(1+\tilde\lambda^2\mu^2)\big)^2\),
which is strictly decreasing in \(|\mu|\). Thus \(M_{\tilde\lambda}(\varphi)=M_\lambda(\varphi\rho)\),
while \(M_\lambda(\rho)=M_{\tilde\lambda}(\mathbf 1)=1\). The difference below is therefore the
covariance of the increasing function \(\varphi\) and the decreasing function \(\rho\) with respect
to the probability measure \(d\nu_\lambda=\tfrac12\nu_\lambda(\mu)\dmu\), since
\begin{align*}
M_{\tilde\lambda}(\varphi)-M_\lambda(\varphi)&=M_\lambda(\varphi\rho)-M_\lambda(\varphi)M_\lambda(\rho)\\
        &=\tfrac12\iint(\varphi(\mu)-\varphi(\mu'))(\rho(\mu)-\rho(\mu'))\,d\nu_\lambda(\mu)\,d\nu_\lambda(\mu'),
\end{align*}
which is negative because the integrand is nonpositive and vanishes only for \(|\mu|=|\mu'|\).
\end{proof}

\begin{lemma}\label{lem:h}
There is a unique \(\lambda_1>0\) with \(h<0\) on \((0,\lambda_1)\) and \(h>0\) on
\((\lambda_1,\infty)\), and \(\lambda_1\in(2,3)\). Consequently \(C_A(w_\lambda)\le C_T(w_\lambda)\)
if and only if \(\lambda\ge\lambda_1\).
\end{lemma}
\begin{proof}
By \eqref{eq:hM} and Lemma~\ref{lem:sym}(i), \(h\) has the sign of
\(\PP(\mu^2)-M_\lambda(\mu^2)\).
By Lemma~\ref{lem:nu}, \(\PP(\mu^2)-M_\lambda(\mu^2)\) is strictly increasing in \(\lambda\).
Hence \(h\) changes sign at most once, from negative to positive. 
Direct evaluation gives \(h(2)<0<h(3)\), using \eqref{eq:sfh},
\(h(2)=\tfrac{13}2\arctan2-\tfrac{37}5\) and \(h(3)=6\arctan3-\tfrac{36}5\).
The sign change
therefore occurs at a unique \(\lambda_1\in(2,3)\). Finally, by
\eqref{eq:CTratio}, \(C_A(w_\lambda)\le C_T(w_\lambda)\) is equivalent to
\(M_\lambda(\mu^2)\le\PP(\mu^2)\), hence to \(h\ge0\), hence to \(\lambda\ge\lambda_1\).
\end{proof}

\begin{lemma}\label{lem:lstar}
Let \(c\in(0,1]\) and let \(\lambda_1\) be as in Lemma~\ref{lem:h}. Then \(\omega_c\) attains its
maximum \(r=\rho_\infty(c)\) on \((0,\infty)\), and every global maximizer \(\lambda=\lambda_*(c)\)
satisfies \(\lambda\ge\lambda_1\). Moreover, \(0<r<1\) and the quantities \(\alpha=1-r\) and
\(A=A(\lambda)\), \(B=B(\lambda)\), \(Y_\infty=Y_\infty(\lambda)\) are related by
\begin{equation}\label{eq:alpharel}
\alpha=\frac{A'}{Y_\infty'}=\frac{(A-B)(3+\lambda^2)^2}{6\lambda^2},\,\,\,\,
\frac rc=A-\alpha Y_\infty,\,\,\,\, \alpha\,Y_\infty(1-Y_\infty)=\tfrac12(A-B).
\end{equation}
\end{lemma}

\begin{proof}
Throughout the proof we use the elementary identities
\begin{equation}\label{eq:AYder}
A'=\frac{B-A}{\lambda},\qquad Y_\infty'=-\frac{6\lambda}{(3+\lambda^2)^2}<0 .
\end{equation}

\emph{\(\omega_c\) is strictly increasing on \((0,\lambda_1]\).} Differentiating
\(\omega_c\), defined in \eqref{eq:omdef}, gives
\begin{equation}\label{eq:Nc}
\omega_c'=\frac{c\big[(A'-Y_\infty')(1-cY_\infty)+cY_\infty'(A-Y_\infty)\big]}{(1-cY_\infty)^2}
=\frac{c\,N_c}{(1-cY_\infty)^2},
\end{equation}
with \(N_c=A'(1-cY_\infty)-Y_\infty'(1-cA)\).
We claim that \(N_c>0\) on \((0,\lambda_1)\) for every \(c\in(0,1]\), which by \eqref{eq:Nc} gives
the assertion.

The function \(N_c\) is affine in \(c\) with \(\partial_cN_c=AY_\infty'-A'Y_\infty\). Inserting
\eqref{eq:AYder} and using \((A-B)(3+\lambda^2)=h+6(1-A)\) from \eqref{eq:sfh},
\[
\partial_cN_c=\frac{-6\lambda^2A+3(A-B)(3+\lambda^2)}{\lambda(3+\lambda^2)^2}
=\frac{3h-6\big(A(3+\lambda^2)-3\big)}{\lambda(3+\lambda^2)^2} .
\]
On \((0,\lambda_1)\) we have \(h<0\) by Lemma~\ref{lem:h}, and \(A(3+\lambda^2)>3\) because
\(A>Y_\infty\) by Lemma~\ref{lem:sym}(i). Both terms in the numerator are therefore negative, so
that \(\partial_cN_c<0\) there and \(N_c\ge N_1\) for every \(c\in(0,1]\). The same two identities
for \(c=1\) give
\begin{align*}
N_1=A'(1-Y_\infty)-Y_\infty'(1-A)
=\frac{\lambda\big[(B-A)(3+\lambda^2)+6(1-A)\big]}{(3+\lambda^2)^2}
=-\frac{\lambda\,h(\lambda)}{(3+\lambda^2)^2}>0,
\end{align*}
on $(0,\lambda_1)$ again by Lemma~\ref{lem:h}. Hence \(N_c\ge N_1>0\) on \((0,\lambda_1)\), which is the claim.

\emph{Existence and \(\lambda\ge\lambda_1\).} By Lemma~\ref{lem:sym}(ii), \(\omega_c\) is
continuous and positive on \((0,\infty)\) and tends to \(0\) as \(\lambda\to\infty\), so there is
\(R>\lambda_1\) with \(\omega_c<\omega_c(\lambda_1)\) on \((R,\infty)\). On the compact interval
\([\lambda_1,R]\) the maximum of \(\omega_c\) is attained, and it is at least
\(\omega_c(\lambda_1)\). Since \(\omega_c<\omega_c(\lambda_1)\) also on \((0,\lambda_1)\) by the
monotonicity just shown, this maximum equals \(r=\rho_\infty(c)\) and every global maximizer
satisfies \(\lambda\ge\lambda_1\).

\emph{The identities \eqref{eq:alpharel}.} A global maximizer is an interior stationary point,
so \((A'-Y_\infty')(1/c-Y_\infty)+(A-Y_\infty)Y_\infty'=0\) by \eqref{eq:Nc}. Dividing by \(Y_\infty'<0\) and using
\(r=\omega_c=c(A-Y_\infty)/(1-cY_\infty)\) gives
\[
r=-\frac{A'-Y_\infty'}{Y_\infty'}=1-\frac{A'}{Y_\infty'},
\]
so that \(\alpha=1-r=A'/Y_\infty'=(A-B)(3+\lambda^2)^2/(6\lambda^2)\) by \eqref{eq:AYder}.
Since \(A-B>0\) by Lemma~\ref{lem:sym}(i), we thus have that \(\alpha>0\), i.e., \(r<1\).
By Lemma~\ref{lem:sym}(ii), \(r=\omega_c(\lambda)>0\).
Solving \(r=c(A-Y_\infty)/(1-cY_\infty)\) for \(c\) gives \(c=r/(A-Y_\infty+rY_\infty)\), that is, \(r/c=A-Y_\infty+rY_\infty=A-\alpha Y_\infty\).
Finally, with \(Y_\infty(1-Y_\infty)=3\lambda^2/(3+\lambda^2)^2\),
\[
\alpha\,Y_\infty(1-Y_\infty)=\frac{(A-B)(3+\lambda^2)^2}{6\lambda^2}\cdot\frac{3\lambda^2}{(3+\lambda^2)^2}
=\tfrac12(A-B).
\]

\end{proof}

We can now prove the main result. Note that the case \(c=0\) immediately gives \(\rho(G)=\rho(G_h)=0\).
\begin{theorem}\label{thm:sharp}
Let \(c \in (0,1]\) and let \(w_\lambda\) be as in \eqref{eq:wlam} with \(\lambda\) being a global maximizer of \(\omega_c\) and
\begin{equation}\label{eq:kappadef}
\kappa=\frac6{(A-B)(3+\lambda^2)} .
\end{equation}
Then \eqref{eq:H} holds with \(r=\rho_\infty(c)\); more precisely, with \(\alpha=1-r\),
\[
\alpha\,C_T(w_\lambda)=1,\qquad C_\partial(w_\lambda)\le C_T(w_\lambda),\qquad C_A(w_\lambda)\le C_T(w_\lambda) .
\]
Consequently, $\rho(G)\le\rho_\infty(c)$ an $\rho(G_h)\le\rho_\infty(c)$.
\end{theorem}

\begin{proof}
Throughout, \(A=A(\lambda)\), \(B=B(\lambda)\), \(Y_\infty=Y_\infty(\lambda)\) and \(\alpha=1-r\) and \(r=\rho_\infty(c)\).
Such a \(\lambda\) exists by Lemma~\ref{lem:lstar}, which also gives \(r,\alpha\in(0,1)\), so that
Theorem~\ref{thm:halfw} is applicable, as well as the three relations \eqref{eq:alpharel} and
\(\lambda\ge\lambda_1\). Consequently, using \eqref{eq:alpharel}, \eqref{eq:momw} and \eqref{eq:kappadef},
\begin{equation}\label{eq:sharpid}
\alpha\kappa=\frac1{1-Y_\infty},
\quad
\PP(w_\lambda g_\lambda)=Y_\infty, 
\quad
\alpha\kappa^2=\frac{6}{\lambda^2(A-B)}.
\end{equation}
We next verify the conditions in \eqref{eq:H}.

\emph{The \(C_T\)-condition} holds with equality, because \eqref{eq:momw2} and
\eqref{eq:sharpid} give
\[
\alpha\,C_T(w_\lambda)=\tfrac{1}6 \alpha\kappa^2\lambda^2(A-B)=1.
\]

\emph{The \(C_\partial\)-condition.} By \eqref{eq:CTratio} and Lemma~\ref{lem:nu}, the ratio
\(C_\partial(w_\lambda)/C_T(w_\lambda)\) is strictly decreasing in \(\lambda\), and by
\eqref{eq:momw2} it equals
\(\tfrac32\big(\ln(1+\lambda^2)+B-1\big)/\big(\lambda^2(A-B)\big)\), which at \(\lambda=2\) is
smaller than one because \(3\ln5<4\arctan2+\tfrac45\). Since \(\lambda\ge\lambda_1>2\) by
Lemma~\ref{lem:h}, we obtain \(C_\partial(w_\lambda)<C_T(w_\lambda)\) and therefore
\(\alpha\,C_\partial(w_\lambda)<\alpha\,C_T(w_\lambda)=1\).

\emph{The \(\QQ\)-condition.} Since \(w_\lambda=\kappa(1-g_\lambda)\), 
we have that \(\operatorname{span}\{\mathbf 1,w_\lambda\}=\operatorname{span}\{\mathbf 1,g_\lambda\}\).
Hence, by Lemma~\ref{lem:coll} it suffices to verify \(\QQ(g,g)\le0\) on \(\mathcal V =\operatorname{span}\{\mathbf 1,g_\lambda\}\). 
The choice of \(w_\lambda\) makes \(g_\lambda\) a null direction of the collision form. We
claim that
\begin{equation}\label{eq:radical}
\QQ(g_\lambda,g)=0\qquad\text{for all }g\in L^2(-1,1).
\end{equation}
Indeed, \(\PP g_\lambda=A\) by Lemma~\ref{lem:mom}(i) and \(\PP(w_\lambda g_\lambda)=Y_\infty\) by
\eqref{eq:sharpid}, so that \(\theta(g_\lambda)=A-\alpha Y_\infty=r/c\) by \eqref{eq:alpharel}.
Therefore, by \eqref{eq:collform},
\begin{align*}
\QQ(g_\lambda,g)&=r\,\theta(g)+r\alpha Y_\infty\PP(w_\lambda g)-r\PP(g_\lambda g)\\
&=r\Big[\PP(g)-\alpha(1-Y_\infty)\PP(w_\lambda g) -\PP(g_\lambda g)\Big]=0,
\end{align*}
because \(w_\lambda=\kappa(1-g_\lambda)\) and \(\alpha\kappa(1-Y_\infty)=1\) by \eqref{eq:sharpid}, proving \eqref{eq:radical}.
Consequently \(\QQ(g,g)=y_1^2\,\QQ(\mathbf 1,\mathbf 1)\) for
\(g=y_1\mathbf 1+y_2g_\lambda\in\mathcal V\), so that the \(\QQ\)-condition holds if and only if
\(\QQ(\mathbf 1,\mathbf 1)\le0\).
By \eqref{eq:radical}, \(\QQ(\mathbf 1,\mathbf 1)=\QQ(f,f)\) for each
\(f=\mathbf 1-\tau g_\lambda\) with \(\tau\in\mathbb R\).
We choose \(\tau\) as
\begin{equation}\label{eq:tau}
\tau=\frac{\PP(w_\lambda)}{\PP(w_\lambda g_\lambda)}=\frac{2(1-A)}{A-B},
\end{equation}
where we used \eqref{eq:momw} in the second step. Hence, \(\PP(w_\lambda f)=0\) and \(\theta(f)=\PP f\), and \eqref{eq:collform} becomes
\[
\QQ(\mathbf 1,\mathbf 1)=\QQ(f,f)=c\,(\PP f)^2-r\,\PP(f^2).
\]
Thus, since \(r/c=A-\alpha Y_\infty\) by \eqref{eq:alpharel},
\(\QQ(\mathbf 1,\mathbf 1)\leq 0\) is equivalent to
\begin{equation}\label{eq:rayleigh}
(\PP f)^2\ \le\ \big(A-\alpha Y_\infty\big)\,\PP(f^2) .
\end{equation}
Denote \(C=A-\alpha Y_\infty\), \(d=A-B>0\) and
\(V_\lambda=\PP(g_\lambda^2)-(\PP g_\lambda)^2=\tfrac12(A+B)-A^2\ge0\), which follows from
Cauchy--Schwarz. We claim that
\begin{equation}\label{eq:Q11}
C\,\PP(f^2)-(\PP f)^2=\frac{V_\lambda\,h(\lambda)}{\lambda^2\,(A-B)} .
\end{equation}
Then \eqref{eq:rayleigh} holds, whence the \(\QQ\)-condition holds, because \(\lambda\ge\lambda_1\)
gives \(h\ge0\) and \(C_A(w_\lambda)\le C_T(w_\lambda)\) by Lemma~\ref{lem:h}.

It remains to prove \eqref{eq:Q11}. Write \(f=(1-\tau)\mathbf 1+\tau(\mathbf 1-g_\lambda)\) and recall from
\eqref{eq:tau} that \(\PP\big((\mathbf 1-g_\lambda)f\big)=\kappa^{-1}\PP(w_\lambda f)=0\). 
Lemma~\ref{lem:mom}(i) then gives for this \(f\)
\begin{equation}\label{eq:f2}
\PP(f^2)=(1-\tau)\,\PP f,\qquad
\PP f=1-\tau A=-\frac{2V_\lambda}{d} ,
\end{equation}
the second equality because \(2V_\lambda=2A(1-A)-d\) and \(\tau=2(1-A)/d\). Consequently
\[
C\,\PP(f^2)-(\PP f)^2=\PP f\,\Big[C(1-\tau)-\PP f\Big],
\]
and it remains to identify the bracket. 
Using \eqref{eq:f2} and \(C-A=-\alpha Y_\infty\),
\[
C(1-\tau)-\PP f=(1-\tau)(C-A)-(1-A)=-(1-\tau)\,\alpha Y_\infty-(1-A) .
\]
With \(s=3+\lambda^2\) we have \(\alpha Y_\infty=ds/(2\lambda^2)\) by \eqref{eq:alpharel} and
\((1-\tau)d=d-2(1-A)\) by \eqref{eq:tau}, so that \(s-\lambda^2=3\) and \eqref{eq:sfh} give
\[
2\lambda^2\Big[C(1-\tau)-\PP f\Big]=-s\big[d-2(1-A)\big]-2\lambda^2(1-A)
=2(1-A)(s-\lambda^2)-ds=-h(\lambda),
\]
and \eqref{eq:Q11} follows from \(\PP f=-2V_\lambda/d\).
Thus \eqref{eq:H} holds, and Theorem~\ref{thm:halfw} gives the bounds for
\(\rho(G)\) and \(\rho(G_h)\).
The admissibility of \(\LL_{w_\lambda}\) on the discrete space is
Lemma~\ref{lem:lam}, which uses only \(\mathbf 1\in Q_N\). 
\end{proof}

\begin{remark}
Since \(\alpha\PP(w_\lambda)(1-Y_\infty)=1-A\) by \eqref{eq:sharpid} and
\(1-Y_\infty=\lambda^2/(3+\lambda^2)\), the inequality \(h(\lambda)\ge0\) is equivalent to
\(\PP(w_\lambda)\le1\), that is, to \(\LL_{w_\lambda}\mathbf 1\le\mathbf 1\), so that the weighted average does not
amplify constants. Equality holds exactly for \(c=1\), where \(\lambda=\lambda_1\),
\(h(\lambda_1)=0\) and \(\PP(w_\lambda)=1\), so that \(\LL_{w_\lambda}\) is a projection.
By Lemma~\ref{lem:h}, \(h\ge0\) is also equivalent to \(C_A(w_\lambda)\le C_T(w_\lambda)\). Thus
the sharp weight makes \(C_T(w_\lambda)\) dominate both competing stability constants.
\end{remark}

\begin{remark}\label{rem:pcg}
Since \(I-S\) and \(I+D\) are self-adjoint and positive with respect to \((\sigma_s\cdot,\cdot)\) by
Lemma~\ref{lem:SD}, conjugate gradients can be applied to \((I-S)\phi=f\) in this inner product
with preconditioner \(I+D\). By Lemma~\ref{lem:lam1} and Theorem~\ref{thm:sharp}, the spectrum
of \((I+D)(I-S)\) lies in \([1-\rho_\infty(c),1]\), so its condition number satisfies
\(\varkappa\le\big(1-\rho_\infty(c)\big)^{-1}\le1.29\). The standard CG estimate in the norm
\(\|\cdot\|_{I-S}=(\sigma_s(I-S)\cdot,\cdot)^{1/2}\) then gives
\begin{equation}\label{eq:pcg}
\|e^{(n)}\|_{I-S}\le 2q^n\,\|e^{(0)}\|_{I-S},\qquad
q=\frac{\sqrt{\varkappa}-1}{\sqrt{\varkappa}+1}\le0.064 .
\end{equation}
The same bounds hold for the discrete iteration.
\end{remark}

\begin{remark}\label{rem:inexact}
If the correction step is carried out inexactly, with \(\widetilde F^{(n+1)}\) in place of
\(F^{(n+1)}=D(I-S)e^{(n)}\), then \(e^{(n+1)}=Ge^{(n)}+F^{(n+1)}-\widetilde F^{(n+1)}\). Hence, if
\(\|F^{(n+1)}-\widetilde F^{(n+1)}\|_\star\le\delta\|e^{(n)}\|_\star\), Theorem~\ref{thm:sharp}
and \(\|G\|_\star=\rho(G)\) give
\[
\|e^{(n+1)}\|_\star\le\big(\rho_\infty(c)+\delta\big)\|e^{(n)}\|_\star ,
\]
The iteration remains a contraction whenever \(\delta<1-\rho_\infty(c)\), 
which quantifies how accurately the diffusion problem has to be solved by an
iterative solver, as is typical in several space dimensions or in reduced precision implementations \cite{MorganOrtega2026}.
\end{remark}

\section{Numerical experiments}\label{sec:num}

We implement the scheme of \cite{PS2020} with continuous \(\mathbb P_1\) finite elements in
\(z\) on \(J\) elements and piecewise constants on \(N\) uniform half-range angular cells. All
spectral radii and constants are computed as (generalized) eigenvalues. The experiments
address, in turn, the discrete bound of Theorem~\ref{thm:sharp}, its sharpness for
heterogeneous media, and the benefit of conjugate gradients.

The test media are piecewise constant on \((0,1)\) and are listed in Table~\ref{tab:media}. All
meshes below contain the material interfaces, so that \(\sigma_t\) and \(\sigma_s\) are constant on
each element and the discrete forms are integrated exactly; that a mesh does not resolve the
boundary or the interface layers means that its cells there are many mean free paths thick, not
that the cross sections are approximated.
As before, \(c\) is the largest layerwise ratio \(\sigma_s/\sigma_t\), which is attained in at
least one layer.
All media of Table~\ref{tab:media} have \(c=0.9999\). 
We use two further families. The periodically layered media consist of \(2P\) layers of equal
width, \(P\in\{1,2,4,8\}\), with \(\sigma_t\) alternating between \(20\) and \(0.2\) and
\(\sigma_s/\sigma_t\) alternating between \(0.999\) and \(c_2\in\{0.999,0\}\), so that
\(c=0.999\); only the cross sections are periodic, and the boundary conditions are inflow as
everywhere below. 
The random media are \(120\) realizations with \(L\in\{1,\dots,5\}\) layers, \(\sigma_t\in[0.03,200]\)
log-uniformly distributed, \(c\in\{0.3,0.6,0.9,0.99,0.999\}\), and in each layer
\(\sigma_s/\sigma_t=c\) or, with probability \(1/2\), uniformly distributed in \([0,c]\).

\begin{table}[htbp]\centering\footnotesize
\caption{Test media on \((0,1)\); the symbol \(|\) separates layers.}
\label{tab:media}
\begin{tabular}{@{}rllll@{}}
\toprule
\# & medium & interfaces & \(\sigma_t\) & \(\sigma_s/\sigma_t\)\\
\midrule
\(1\) & homogeneous, thick & --- & \(10^3\) & \(0.9999\)\\
\(2\) & thick \(|\) thin (void-like) & \(0.5\) & \(10^3\,|\,10^{-2}\) & \(0.9999\,|\,0\)\\
\(3\) & thick \(|\) thin (scatterer) & \(0.5\) & \(10^3\,|\,1\) & \(0.9999\,|\,0.9999\)\\
\(4\) & thin \(|\) thick \(|\) thin & \(0.3,\,0.7\) & \(1\,|\,10^3\,|\,1\) & \(0.99\,|\,0.9999\,|\,0.99\)\\
\(5\) & thick \(|\) absorber \(|\) thick & \(0.45,\,0.55\) & \(10^3\,|\,10^2\,|\,10^3\) & \(0.9999\,|\,0\,|\,0.9999\)\\
\(6\) & thick \(|\) thin \(|\) thick & \(0.45,\,0.55\) & \(10^3\,|\,1\,|\,10^3\) & \(0.9999\) in all layers\\
\(7\) & thick \(|\) void-like \(|\) thick & \(0.45,\,0.55\) & \(10^3\,|\,10^{-3}\,|\,10^3\) & \(0.9999\,|\,0\,|\,0.9999\)\\
\(8\) & jump in \(\sigma_t\) only & \(0.5\) & \(10^3\,|\,10\) & \(0.9999\,|\,0.9999\)\\
\(9\) & jump in \(\sigma_s/\sigma_t\) only & \(0.5\) & \(10^2\,|\,10^2\) & \(0.9999\,|\,0.5\)\\
\bottomrule
\end{tabular}
\end{table}

\paragraph{The bound \(\rho_\infty(c)\)}
Table~\ref{tab:res} shows \(\rho(G_h)\) for two optically thick slabs under refinement in \(z\) and
in \(\mu\). The meshes are graded, with the nodes of each layer at the distances
\(g^j/(3\sigma_t)\), \(j\ge0\), from both ends, with the local \(\sigma_t\) and a cap at one
eighth of the layer width. The cells at the boundaries and interfaces are therefore a fraction
of a mean free path and grow by the factor \(g\) towards the middle, so that \(g\) is the
refinement parameter in \(z\). As expected from Theorem~\ref{thm:sharp}, the rates approach
\(\rho_\infty(c)\) from below as the resolution increases, mainly under angular refinement.
Uniform meshes with the same number of cells do not resolve the least-damped modes, whose
wavelength is \(2\pi/\lambda_*(c)\approx2.5\) mean free paths by \eqref{eq:alsymbol}; for
medium~\(1\) at \(N=64\) their cells span \(19\), \(6\) and \(3\) mean free paths and the rates drop to
\(0.029\), \(0.080\) and \(0.148\). A uniform mesh returns \(0.22458\) only for \(h\sigma_t\le5/4\),
that is \(J\ge800\); the grading merely provides cells of about one mean free path in the thick
region.

\begin{table}[htbp]\centering\footnotesize
\caption{Resolution study, \(c=0.9999\): \(\rho(G_h)\) for the media \(1\) and \(6\) of Table~\ref{tab:media}, for mesh grading
ratios \(g\). Here \(\rho_\infty(c)=0.22463\) and
\(c/4=0.24998\).}
\label{tab:res}
\begin{tabular}{@{}lrcccc@{}}
\toprule
medium & \(J\) & \(N=16\) & \(N=32\) & \(N=64\) & \(N=128\)\\
\midrule
medium \(1\), \(g=1.35\) & \(52\)  & \(0.22167\) & \(0.22233\) & \(0.22250\) & \(0.22254\)\\
\phantom{medium \(1\)}, \(g=1.1\) & \(156\) & \(0.22354\) & \(0.22419\) & \(0.22435\) & \(0.22440\)\\
\phantom{medium \(1\)}, \(g=1.05\) & \(302\) & \(0.22370\) & \(0.22435\) & \(0.22452\) & \(0.22456\)\\
medium \(6\), \(g=1.35\) & \(104\) & \(0.22372\) & \(0.22438\) & \(0.22455\) & \(0.22459\)\\
\phantom{medium \(6\)}, \(g=1.1\) & \(312\) & \(0.22373\) & \(0.22439\) & \(0.22456\) & \(0.22460\)\\
\phantom{medium \(6\)}, \(g=1.05\) & \(600\) & \(0.22376\) & \(0.22441\) & \(0.22458\) & \(0.22462\)\\
\bottomrule
\end{tabular}
\end{table}

Table~\ref{tab:half} reports \(\rho(G_h)\) for \(40\) random
heterogeneous media per scattering ratio \(c\), with one to five layers, \(\sigma_t\in[0.03,1000]\)
log-uniformly distributed, random meshes whose interior nodes are the layer interfaces together
with \(J-1\) independent uniform points in \((0,1)\), \(J\in\{2,\dots,47\}\), so that the cells at
the boundaries and interfaces are many mean free paths thick, and \(N\) drawn from
\(\{1,2,3,4,8,16,24\}\); no violation of Theorem~\ref{thm:sharp} occurred.
For \(N=1\) the diffusion subspace exhausts the discrete space, so that the correction is exact
and \(\rho(G_h)=0\). 
Let \(q_h\) denote the maximum over all \(v\in\Wph\) of the ratio of the left-hand side over
the right-hand side of \eqref{eq:fenchel} for \(\chi=\LL_{w_\lambda}v\) with the weight \eqref{eq:wlam}, \eqref{eq:kappadef}.
In all cases \(q_h\le1\), as asserted by the proof. 
Its proximity to one reflects that the \(C_T\)-condition in \eqref{eq:H} holds with equality,
so that the estimate of the transport term in the proof of Theorem~\ref{thm:halfw} is
saturated; it does not indicate that the global bound \(\rho_\infty(c)\) is attained.

\begin{table}[htbp]\centering\footnotesize
\caption{Theorem~\ref{thm:sharp} on \(40\) random heterogeneous media per row: largest computed
\(\rho(G_h)\), the bound \(\rho_\infty(c)\), the rate \(c/4\) of Remark~\ref{rem:half}, and
the largest constant \(q_h\).}
\label{tab:half}
\begin{tabular}{@{}ccccc@{}}
\toprule
\(c\) & \(\max\rho(G_h)\) & \(\rho_\infty(c)\) & \(c/4\) & \(\max q_h\)\\
\midrule
0.5000 & 0.0931 & 0.0958 & 0.1250 & 0.99918\\
0.9000 & 0.1918 & 0.1939 & 0.2250 & 0.99945\\
0.9900 & 0.2190 & 0.2214 & 0.2475 & 0.99959\\
0.9999 & 0.2206 & 0.2246 & 0.2500 & 0.99998\\
\bottomrule
\end{tabular}
\end{table}

\paragraph{Sharpness for heterogeneous media} 
Table~\ref{tab:cmp} provides numerical evidence that the bound of Theorem~\ref{thm:sharp} is sharp.
The media of Table~\ref{tab:media} separate heterogeneity in \(\sigma_t\) from heterogeneity in \(\sigma_s/\sigma_t\), and include thin,
void-like, absorbing, periodic and randomly layered regions.
 We use the finest resolution of Table~\ref{tab:res}, \(g=1.05\) and \(N=64\).
For the periodically layered and the random media we use \(g=1.1\) and \(N=64\). 
In every case the computed rate remains below \(\rho_\infty(c)\). 
Every medium with an optically thick layer in which \(c\) is attained reaches \(\rho_\infty(c)\) up to
the discretization error, whether the other layers are thin, void-like, absorbing or
scattering. In these examples the rate is essentially determined by that layer.
The periodically layered media, whose scattering layers have optical thickness at most \(10\),
stay slightly below, with \(\rho(G_h)\)
between \(0.22275\) and \(0.22292\) against \(\rho_\infty(0.999)=0.22434\), and over the
random ensemble the largest ratio \(\rho(G_h)/\rho_\infty(c)\) is \(0.9995\). For \(60\)
further random configurations on slabs \((0,Z)\) with \(\sigma_t=1\) and \(\sigma_s=c\), so that
\(Z\in[0.1,200]\) is the optical thickness, with \(c\in[0.9,0.9999]\),
\(N\in\{2,\dots,24\}\) angular cells and uniform, random or strongly graded meshes with
\(J\in\{4,\dots,64\}\) elements,
the largest ratio was \(0.9977\). These are numerical observations, which neither assert that
\(\rho(G_h)=\rho_\infty(c)\) for a given slab, nor that the weight \eqref{eq:wlam} is
optimal among all admissible weights.

\begin{table}[htbp]\centering\footnotesize
\caption{Computed \(\rho(G_h)\) for the media of Table~\ref{tab:media}, for which
\(\rho_\infty(c)=0.22463\).}
\label{tab:cmp}
\setlength{\tabcolsep}{4.5pt}
\begin{tabular}{@{}l ccccccccc@{}}
\toprule
medium & \(1\) & \(2\) & \(3\) & \(4\) & \(5\) & \(6\) & \(7\) & \(8\) & \(9\)\\
\midrule
\(\rho(G_h)\) & \(0.22452\) & \(0.22452\) & \(0.22452\) & \(0.22452\) & \(0.22452\) & \(0.22458\) & \(0.22458\)
& \(0.22454\) & \(0.22449\)\\
\bottomrule
\end{tabular}
\end{table}

\paragraph{Conjugate gradients} Table~\ref{tab:solvers} compares the DSA iteration
with DSA preconditioned conjugate gradients for the smooth medium
\(\sigma_s=\bar\sigma_s/\eps\), \(\sigma_a=\eps\bar\sigma_a\) of \cite[\S6.3]{PS2020}. Each
iteration of either method requires one transport sweep and one application of the diffusion
preconditioner. The computed condition numbers are close to \(1/(1-\rho(G_h))\) and stay below
the bound \(1.29\) of Remark~\ref{rem:pcg}, which they reach in the scattering dominated cases.
Accordingly, conjugate gradients need at most seven iterations, and reduce the number of sweeps
by roughly one third to one half.

\begin{table}[htbp]\centering\footnotesize
\caption{Solver comparison for the smooth medium with \(\bar\sigma_s=1+\tfrac12\sin2\pi z\),
\(\bar\sigma_a=1+\tfrac14\cos\pi z\) on \((0,1)\), uniform mesh, \(J=64\), \(N=16\), random exact
scalar flux, errors in the norm \(\|\cdot\|_{I-S}\): spectral radius
\(\rho(G_h)\) of the DSA iteration, condition number \(\varkappa\) and bound \(q\) of
\eqref{eq:pcg} for DSA preconditioned CG, and numbers of iterations for an
error reduction by \(10^{-8}\).}
\label{tab:solvers}
\begin{tabular}{@{}rcc|ccc@{}}
\toprule
& \multicolumn{2}{c|}{DSA iteration} & \multicolumn{3}{c}{DSA-PCG}\\
\(\eps\) & \(\rho(G_h)\) & its & \(\varkappa\) & rate & its\\
\midrule
\(1\)       & \(0.0937\) & \(8\)  & \(1.1025\) & \(0.0244\) & \(5\)\\
\(10^{-1}\) & \(0.2194\) & \(11\) & \(1.2629\) & \(0.0583\) & \(7\)\\
\(10^{-2}\) & \(0.2237\) & \(12\) & \(1.2876\) & \(0.0631\) & \(7\)\\
\(10^{-3}\) & \(0.0415\) & \(6\)  & \(1.0433\) & \(0.0106\) & \(4\)\\
\(10^{-4}\) & \(0.0038\) & \(4\)  & \(1.0038\) & \(0.0009\) & \(3\)\\
\bottomrule
\end{tabular}
\end{table}

For \(\eps\le10^{-3}\) the cells of the uniform mesh are many mean free paths thick. The
discrete space then contains neither the modes of wavelength \(2\pi/\lambda_*(c)\approx2.5\) mean
free paths, which converge slowest for the continuous iteration, nor the kinetic boundary layers
of width \(O(\eps)\). This explains the small rates in the last two rows. On
meshes that resolve these layers the rate is close to \(\rho_\infty(c)\), cf.\
Table~\ref{tab:cmp}. This is a
discretization effect and not a property of the continuous iteration. Rates that decrease as
the cells become optically thick were also observed for even-parity \(S_N\) DSA in
\cite{MorelMcGhee1995}, where the computed rates are moreover found to be sensitive to
the diffusion boundary extrapolation length; in the variational form this length is fixed by
\eqref{eq:corr}. Theorem~\ref{thm:sharp} covers both situations.

\section{Conclusions}\label{sec:concl}

We have studied the DSA source iteration in slab geometry with inflow boundary conditions and
variable bounded cross sections with \(\sigma_t-\sigma_s\ge\gamma>0\), together with its
Galerkin discretization of \cite{PS2020} on every conforming tensor-product space whose angular
factor contains the constants. For both we have proved the uniform bounds
\[
\rho(G)\le\rho_\infty(c),
\qquad
\rho(G_h)\le\rho_\infty(c).
\]
 The value \(\rho_\infty(c)\), which Fourier analysis yields for an infinite homogeneous medium, is thus a uniform bound for heterogeneous slabs and for their discretizations. The estimates are independent of the slab thickness, the spatial mesh, and the angular resolution. Moreover, the condition number of the preconditioned system is bounded by \(1+0.29\,c\), yielding rapid convergence of the DSA iteration and of the corresponding preconditioned conjugate gradient method.
In the terminology of \cite{LarsenMorel2010}, this establishes unconditional effectiveness of the variational DSA iteration for heterogeneous slabs. 
Azmy \cite{Azmy2002} left open the corresponding question for cell-centered preconditioners of weighted-difference schemes.

Our framework suggests extensions to higher-dimensional transport and anisotropic scattering. 
In several space dimensions, Lemmas~\ref{lem:SD}--\ref{lem:primal} hold verbatim, but the local weighted projection no longer controls the full spatial gradient, since \(s\cdot\nabla v\) is only a directional derivative. 
A bound of the present form would therefore require a spatially nonlocal interpolant into the diffusion subspace. 
For anisotropic scattering, the scattering operator is no longer a projection, although similar arguments may remain possible in the weakly anisotropic case.

DSA is not restricted to source problems. Linear diffusion acceleration for the \(k\)-eigenvalue problem was studied in \cite{BarbuAdams2022}, and DSA has also been combined with Anderson and Chebyshev acceleration of the power iteration \cite{CallooEtAl2023}. Since the inner iteration is the accelerated source iteration studied here, the bounds obtained above may also be useful in these settings.

\section*{Declaration on the use of AI tools}
Generative AI tools were used in the preparation of this manuscript for drafting and revising portions of the exposition, as well as for language editing and proofreading. 
The author reviewed and verified all mathematical statements, proofs, numerical results, and references and take full responsibility for the content of the manuscript.

\bibliographystyle{siamplain}
\bibliography{references}

\end{document}